\documentclass{amsart}
\usepackage{amssymb}
\usepackage{amsmath}
\usepackage{amsfonts}

\newtheorem{theorem}{Theorem}
\theoremstyle{plain}

\newtheorem{corollary}{Corollary}

\newtheorem{definition}{Definition}
\newtheorem{example}{Example}

\newtheorem{lemma}{Lemma}

\newtheorem{proposition}{Proposition}
\newtheorem{remark}{Remark}

\numberwithin{equation}{section}
\input{tcilatex}

\begin{document}
\title{ALMOST COSYMPLECTIC STATISTICAL MANIFOLDS}
\author{Aziz Yazla$^{1}$}
\address{ $1-$Uludag University, Science Institute, Gorukle 16059,
Bursa-TURKEY}
\email{501411002@ogr.uludag.edu.tr, cengiz@uludag.edu.tr}
\author{ \.{I}rem K\"{u}peli Erken$^{2}$}
\address{$2-$Faculty of Natural Sciences, Architecture and Engineering,
Department of Mathematics, Bursa Technical University, Bursa-TURKEY}
\email{ irem.erken@btu.edu.tr}
\author{Cengizhan Murathan $^{3}$}
\address{$3-$Art and Science Faculty, Department of Mathematics, Uludag
University, Gorukle 16059, Bursa-TURKEY}
\email{ cengiz@uludag.edu.tr}

\begin{abstract}
This paper is a study of almost contact statistical manifolds. Especially
this study is focused on almost cosymplectic statistical manifolds. We
obtained basic properties of such manifolds. It is proved a characterization
theorem and a corollary for the almost cosymplectic statistical manifold
with Kaehler leaves. We also study curvature properties of an almost
cosymplectic statistical manifold. Moreover, examples are constructed.
\end{abstract}

\maketitle


\section{I\textbf{ntroduction}}

\bigskip Let $(M,g)$ be a Riemannian manifold and $\nabla $ be an affine
connection on $M$. An affine connection $\nabla ^{\ast }$ is called a
conjugate (dual) connection of $\nabla $ if%
\begin{equation}
Zg(X,Y)=g(\nabla _{Z}X,Y)+g(X,\nabla _{Z}^{\ast }Y)  \label{STAT1}
\end{equation}%
for any $X,Y,Z\in \Gamma (M).$ In this situation $\nabla g$ is symmetric.
The triple $(g,\nabla ,\nabla ^{\ast })$ is called a dualistic structure on $%
M$ and the quadruplet $(M,g$,$\nabla $,$\nabla ^{\ast }$) is called
statistical manifold. A. P. Norden introduced these connections to affine
differential geometry and then\ U. Simon gave an excellent survey concerning
the notion of \textquotedblleft conjugate connection" ( \cite{SI}). The
notion of conjugate connection was first initiated into statistics by Amari 
\cite{Amari}. In his studies involves statistical problems getting some
solutions and developed by Lauritzen \cite{LA}. If $\nabla $ coincides with $%
\nabla ^{\ast }$ then statistical manifold simply reduces to usual
Riemannian manifold. Clearly, $(\nabla ^{\ast })^{\ast }$ $=\nabla $. In a
sense, duality is involutive. One can also show that $2\nabla ^{0}=\nabla
+\nabla ^{\ast }$, where $\nabla ^{0}$ is Riemannian connection with respect
to $g$. \ In \cite{KU}, T.Kurose studied affine immersions of statistical
manifolds into the affine space and noticed that there is a close
relationship between the geometry of statistical manifolds and affine
geometry. Otherhand Lagrangian submanifolds of complex space forms are also
naturally endowed with statistical structures (see \cite{SHIMA} pp. 34). So
statistical manifolds play an important role in differential geometry.

Recently, H.Furuhuta \cite{FURUHATA1} defined and studied the holomorphic
statistical manifold which can be considered as a Kaehler manifold with a
certain connection. Then holomorphic statistical manifold notion is expanded
to the statistical counterparts of a Sasakian manifold and a Kenmotsu
manifold by \cite{FURUHATA2}, \cite{FURUHATA3} and \cite{MM}. \ Other hand
K. Takano \cite{TAKANO},\cite{KTAKANO} defined Kaehler-like and Sasaki-like
statistical manifolds which are considered \ setting suitable complex
structures and suitable contact structures on statistical manifolds. These\
studies motivate us to study on almost complex statistical and almost
contact statistical manifolds. Especially our main purpose here is to extend
these results to almost cosymplectic statistical manifolds.

In the present paper, we are interested in almost Hermitian statistical
manifolds and almost contact statistical manifolds which include Kaehler,
Sasakian, Kenmotsu and cosymplectic statistical manifolds. The paper is
organized as follows. In Section 2, we provide a brief the notions of
statistical manifolds and almost contact manifolds. In Section 3, we define
almost Hermitian statistical manifolds and \ study almost Kaehler
statistical manifolds. In Section 4, we introduce almost contact statistical
manifolds and provide a few basic equalities. In Section 5, we study almost
cosymplectic statistical manifolds and give a characterization, almost
Hermitian manifolds and almost contact manifolds. In Section 6, we give also
a characterization of almost cosymplectic statistical manifolds with Kaehler
statistical leaves and provide examples on almost cosymplectic statistical
manifolds. In the last section, we study curvature properties of an almost
cosymplectic statistical manifold.

\section{Preliminaries}

For a statistical manifold $(M,g$,$\nabla $,$\nabla ^{\ast }$) the
difference $(1,2)$ tensor $\mathcal{K}$ of \ a torsion free affine
connection $\nabla $ and Levi-Civita connection $\nabla ^{0}$ is defined as%
\begin{equation}
\mathcal{K}_{X}Y=\mathcal{K}(X,Y)=\nabla _{X}Y-\nabla _{X}^{0}Y.  \label{K1}
\end{equation}%
Because of $\nabla $ and $\nabla ^{0}$ are torsion free, we have%
\begin{equation}
\mathcal{K}_{X}Y=\mathcal{K}_{Y}X,\text{ \ \ }\ g(\mathcal{K}_{X}Y,Z)=g(Y,%
\mathcal{K}_{X}Z)  \label{K2}
\end{equation}%
for any $X,Y,Z\in \Gamma (TM)$. By (\ref{STAT1}) and (\ref{K1}), one can
obtain 
\begin{equation}
\mathcal{K}_{X}Y=\nabla _{X}^{0}Y-\nabla _{X}^{\ast }Y.  \label{K3}
\end{equation}%
Using (\ref{K1}) and (\ref{K3}), we find 
\begin{equation}
2\mathcal{K}_{X}Y=\nabla _{X}Y-\nabla _{X}^{\ast }Y.  \label{K4}
\end{equation}%
By (\ref{K1}), we have 
\begin{equation}
g(\nabla _{X}Y,Z)=g(\mathcal{K}_{X}Y,Z)+g(\nabla _{X}^{0}Y,Z).  \label{K5}
\end{equation}%
An almost Hermitian manifold $(N^{2n},g,J)$ is a smooth manifold endowed
with an almost complex structure $J$ and a Riemannian metric $g$ compatible
in the sense

\begin{equation*}
J^{2}X=-X,\text{ }g(JX,Y)=-g(X,JY)
\end{equation*}%
for any $X,Y\in \Gamma (TN).$ The fundamental $2$-form $\Omega $ of an
almost Hermitian manifold is defined by%
\begin{equation*}
\Omega (X,Y)=g(JX,Y)
\end{equation*}%
for any vector fields $X,Y$ on $N$. An almost Hermitian manifold is called
an almost Kaehler manifold if its fundamental form $\Omega $ is closed, that
is, $d\Omega $ $=0.$ If Nijenhuis torsion of $J$ satisfies%
\begin{equation*}
N_{J}(X,Y)=[X,Y]-[JX,JY]+J[X,JY]+J[JX,Y]=0
\end{equation*}%
then $(N^{2n},g,J)$ is called Kaehler manifold. It is also well known that
an almost Hermitian manifold $(M,J,g)$ is Kaehler if and only if its almost
complex structure $J$ is parallel with respect to the Levi-Civita connection 
$\nabla ^{0}$, that is, $\nabla ^{0}J=0$ (\cite{KO}).

Let $M$ be a $(2n+1)$-dimensional differentiable manifold and $\phi $ is a $%
(1,1)$ tensor field, $\xi $ is a vector field and $\eta $ is a one-form on $%
M.$ If $\phi ^{2}=-Id+\eta \otimes \xi ,\quad \eta (\xi )=1$ then $(\phi
,\xi ,\eta )$ is called an almost contact structure on $M$ . The manifold $M$
is said to be an almost contact manifold if it is endowed with an almost
contact structure \cite{Blair}.

Let $M$ be an almost contact manifold. $M$ will be called an almost contact
metric manifold if it is additionally endowed with a Riemannian metric $g$ ,
i.e.%
\begin{equation}
g(\phi X,\phi Y)=g(X,Y)-\eta (X)\eta (Y).  \label{1}
\end{equation}

For such manifold, we have 
\begin{equation}
\eta (X)=g(X,\xi ),\text{ }\phi (\xi )=0,\text{ }\eta \circ \phi =0.
\label{2}
\end{equation}

Moreover, we can define a skew-symmetric tensor field (a $2$-form) $\Phi $ by%
\begin{equation}
\Phi (X,Y)=g(\phi X,Y),  \label{3}
\end{equation}%
usually called fundamental form.$^{{}}$

~On an almost contact manifold, the $(1,2)$-tensor field $N^{(1)}$ is
defined by%
\begin{equation*}
N^{(1)}(X,Y)=\left[ \phi ,\phi \right] (X,Y)-2d\eta (X,Y)\xi ,
\end{equation*}%
where $\left[ \phi ,\phi \right] $ is the Nijenhuis torsion of $\phi $%
\begin{equation*}
\left[ \phi ,\phi \right] (X,Y)=\phi ^{2}\left[ X,Y\right] +\left[ \phi
X,\phi Y\right] -\phi \left[ \phi X,Y\right] -\phi \left[ X,\phi Y\right] .
\end{equation*}

If $N^{(1)}$ vanishes identically, then the almost contact manifold
(structure) is said to be normal \cite{Blair}. The normality condition says
that the almost complex structure $J$ defined on $M\times 
\mathbb{R}
$%
\begin{equation*}
J(X,\lambda \frac{d}{dt})=(\phi X+\lambda \xi ,\eta (X)\frac{d}{dt}),
\end{equation*}%
is integrable.

An almost contact metric manifold $M^{2n+1}$, with a\ structure $(\phi ,\xi
,\eta ,g)$ is said to be an almost cosymplectic manifold, if 
\begin{equation}
d\eta =0,\quad d\Phi =0.  \label{4}
\end{equation}%
If additionally normality conditon is fulfilled, then manifold is called
cosymplectic.

On the other hand, Kenmotsu studied in \cite{KENMOTSU} another class of
almost contact manifolds, defined by the following conditions on the
associated almost contact structure%
\begin{equation}
d\eta =0,\quad d\Phi =2\eta \wedge \Phi .  \label{5}
\end{equation}%
A normal almost Kenmotsu manifold is said to be a Kenmotsu manifold.

When 
\begin{equation}
d\eta =\Phi   \label{6}
\end{equation}%
an almost contact manifold is called a contact metric manifold \cite{Blair}.
A contact metric manifold $M^{2n+1}$ is a Sasakian manifold if the structure
is normal.

\section{Dualistic structure on almost Hermitian manifolds}

\begin{definition}
Let ($N^{2n},g,\nabla ,\nabla ^{\ast })$ be a statistical manifold. If $%
(N^{2n},g,J)$ $\ $is an almost Hermitian manifold then ($N^{2n},g,J,\nabla
,\nabla ^{\ast })$ is called almost Hermitian statistical manifold. If $%
(N^{2n},g,J)$ $\ $is an (almost) Kaehler manifold then ($N^{2n},g,J,\nabla
,\nabla ^{\ast })$ is called (almost) Kaehler statistical manifold.
\end{definition}

After some calculations one can easily get following.

\begin{lemma}[ \protect\cite{NODA}]
\label{Conjugate J}Let ($N^{2n},g,\nabla ,\nabla ^{\ast })$ be an almost
Hermitian statistical manifold. Then the following equation 
\begin{equation}
g((\nabla _{X}J)Y,Z)=-g(Y,(\nabla _{X}^{\ast }J)Z)  \label{AZIZ1}
\end{equation}%
holds for any $X,Y,Z\in \Gamma (TM).$
\end{lemma}

Using (\ref{K1}) and (\ref{K2}), we easily find the next result.

\begin{lemma}
\label{Conjugate JJ}Let ($N^{2n},g,\nabla ,\nabla ^{\ast })$ be an almost
Hermitian statistical manifold. Then%
\begin{eqnarray}
(\nabla _{X}J)Y &=&(\nabla _{X}^{0}J)Y+(\mathcal{K}_{X}J)Y,  \label{AZIZ2} \\
(\nabla _{X}^{\ast }J)Y &=&(\nabla _{X}^{0}J)Y-(\mathcal{K}_{X}J)Y.
\label{AZIZ3}
\end{eqnarray}%
for any $X,Y\in \Gamma (TM).$
\end{lemma}

\begin{lemma}
\label{FUNDAMENTAL FORM 2}For an almost Hermitian statistical manifold we
have%
\begin{equation}
(\nabla _{X}\Omega )(Y,Z)=g((\nabla _{X}J)Y,Z)-2g(\mathcal{K}_{X}JY,Z),
\label{AZIZ4}
\end{equation}%
and%
\begin{equation}
(\nabla _{X}^{\ast }\Omega )(Y,Z)=g((\nabla _{X}^{\ast }J)Y,Z)+2g(\mathcal{K}%
_{X}JY,Z)  \label{AZIZ5}
\end{equation}%
for any $X,Y,Z\in \Gamma (TM).$
\end{lemma}

\begin{proof}
According to a vector field, the derivative of 2-form $\Phi $ can be written
as 
\begin{equation*}
(\nabla _{X}\Omega )(Y,Z)=X\Omega (Y,Z)-\Omega (\nabla _{X}Y,Z)-\Omega
(Y,\nabla _{X}Z).
\end{equation*}%
By (\ref{STAT1}), (\ref{K1}) and (\ref{K2}), we obtain 
\begin{eqnarray*}
(\nabla _{X}\Omega )(Y,Z) &=&Xg(JY,Z)-g(J\nabla _{X}Y,Z)-g(JY,\nabla _{X}Z)
\\
&=&g(\nabla _{X}JY,Z)+g(JY,\nabla _{X}^{\ast }Z)-g(J\nabla
_{X}Y,Z)-g(JY,\nabla _{X}Z) \\
&=&g((\nabla _{X}J)Y,Z)-2g(\mathcal{K}_{X}JY,Z).
\end{eqnarray*}%
This leads to (\ref{AZIZ4}). If we similarly calculate the derivative of
2-form $\Omega $ respect to conjugate connection $\nabla ^{\ast }$, we get (%
\ref{AZIZ5})$.$
\end{proof}

By Lemma \ref{Conjugate JJ} and the relations (\ref{K1}) and (\ref{K2}) we
easily prove following corollary.

\begin{corollary}
\label{FUNDAMENTAL FORM ULUDAG}For an almost Hermitian statistical manifold
we have 
\begin{equation}
(\nabla _{X}\Omega )(Y,Z)=(\nabla _{X}^{0}\Omega )(Y,Z)-g(\mathcal{K}_{X}JY+J%
\mathcal{K}_{X}Y,Z)  \label{AZIZ5A}
\end{equation}%
and%
\begin{equation}
(\nabla _{X}^{\ast }\Omega )(Y,Z)=(\nabla _{X}^{0}\Omega )(Y,Z)+g(\mathcal{K}%
_{X}JY+J\mathcal{K}_{X}Y,Z)  \label{AZIZ5B}
\end{equation}%
for any $X,Y,Z\in \Gamma (TM).$
\end{corollary}

\begin{theorem}
\label{FUNDAMENTAL FORM 3}Let ($N^{2n},g,,J,\nabla ,\nabla ^{\ast })$ be an
almost Hermitian statistical manifold. The covariant derivatives $\nabla
J,\nabla ^{\ast }J$ of $J$ with respect to the torsion free connections $%
\nabla $ and $\nabla ^{\ast }$ are given by%
\begin{eqnarray}
2g((\nabla _{X}J)Y,Z) &=&2g((\mathcal{K}_{X}J)Y,Z)  \label{AZIZ6} \\
&&+3d\Omega (X,Y,Z)-3d\Omega (X,JY,JZ)+g(N_{J}(Y,Z),JX),  \notag \\
2g((\nabla _{X}^{\ast }J)Y,Z) &=&-2g((\mathcal{K}_{X}J)Y,Z)  \label{AZIZY7}
\\
&&+3d\Omega (X,Y,Z)-3d\Omega (X,JY,JZ)+g(N_{J}(Y,Z),JX)  \notag
\end{eqnarray}%
for any $X,Y,Z\in \Gamma (TM).$
\end{theorem}

\begin{proof}
It is well known that the covariant derivative $\nabla ^{0}J$ satisfies 
\begin{equation}
2g((\nabla _{X}^{0}J)Y,Z)=3d\Omega (X,Y,Z)-3d\Omega
(X,JY,JZ)+g(N_{J}(Y,Z),JX),  \notag
\end{equation}%
for any $X,Y,Z\in \Gamma (TM).$ If we notice the relations (\ref{AZIZ2}) and
(\ref{AZIZ3}) we reach to our equations.
\end{proof}

\begin{corollary}
\label{KAEHLER}Let ($N^{2n},g,J,\nabla ,\nabla ^{\ast })$ be an almost
Kaehler statistical manifold. Then%
\begin{eqnarray}
2g((\nabla _{X}J)Y,Z) &=&2g((\mathcal{K}_{X}J)Y,Z)+g(N_{J}(Y,Z),JX)
\label{AZIZ8} \\
2g((\nabla _{X}^{\ast }J)Y,Z) &=&-2g((\mathcal{K}_{X}J)Y,Z)+g(N_{J}(Y,Z),JX)
\label{AZIZ9}
\end{eqnarray}%
\textit{for any} $X,Y,Z\in \Gamma (TM).$
\end{corollary}

By (\ref{K2}) we can give following.

\begin{proposition}
\label{FUNDAMENTAL FORM 86}Let $(M^{n},g,\nabla ,\nabla ^{\ast })$ be a
statistical manifold and $\psi $ be a skew symmetric $(1,1)$ tensor field on 
$M$. \ Then we have%
\begin{equation}
g(\mathcal{K}_{X}\psi Y+\psi \mathcal{K}_{X}Y,Z)+g(\mathcal{K}_{Z}\psi
X+\psi \mathcal{K}_{Z}X,Y)+g(\mathcal{K}_{Y}\psi Z+\psi \mathcal{K}_{Y}Z,X)=0
\label{AZIZ IDENTITIY}
\end{equation}%
for any $X,Y,Z\in \Gamma (TM).$
\end{proposition}

\begin{corollary}
\label{FUNDAMENTAL FORM BURSA}Let ($N^{2n},g,J,\nabla ,\nabla ^{\ast })$ be
an almost Kaehler statistical manifold. Then%
\begin{eqnarray}
(\nabla _{X}\Omega )(Y,Z)+(\nabla _{Z}\Omega )(X,Y)+(\nabla _{Y}\Omega
)(Z,X) &=&0,  \label{AZIZ81} \\
(\nabla _{X}^{\ast }\Omega )(Y,Z)+(\nabla _{Z}^{\ast }\Omega )(X,Y)+(\nabla
_{Y}^{\ast }\Omega )(Z,X) &=&0  \label{AZIZ82}
\end{eqnarray}%
for any $X,Y,Z\in \Gamma (TM)$ .
\end{corollary}

\begin{proof}
Since $d\Omega =0$ we have%
\begin{equation}
(\nabla _{X}^{0}\Omega )(Y,Z)+(\nabla _{Z}^{0}\Omega )(X,Y)+(\nabla
_{Y}^{0}\Omega )(Z,X)=0.  \label{AZIZ83}
\end{equation}
Using (\ref{AZIZ2}), Corollary \ref{FUNDAMENTAL FORM ULUDAG} and Proposition %
\ref{FUNDAMENTAL FORM 86} we have requested equations.
\end{proof}

\begin{corollary}
\label{KAEHLER2}Let ($N^{2n},g,J,\nabla ,\nabla ^{\ast })$ be a Kaehler
statistical manifold. Then%
\begin{eqnarray}
g((\nabla _{X}J)Y,Z) &=&g((\mathcal{K}_{X}J)Y,Z)  \label{AZIZ10} \\
g((\nabla _{X}^{\ast }J)Y,Z) &=&-g((\mathcal{K}_{X}J)Y,Z)  \label{AZIZ11}
\end{eqnarray}%
\textit{for any} $X,Y,Z\in \Gamma (TM).$
\end{corollary}

\begin{definition}[\protect\cite{FURUHATA1},\protect\cite{FURUHATA2}]
Let $(M^{n},g,\nabla ,\nabla ^{\ast })$ be statistical manifold. A $2$-form $%
\omega $ on $M^{n}$ is defined by%
\begin{equation}
\omega (X,Y)=g(\psi X,Y)  \label{AZIZ12}
\end{equation}%
where $\psi $ is skew symmetric $(1,1)$ tensor field on $M.$ If $\mathcal{K}%
_{X}\psi Y+\psi \mathcal{K}_{X}Y=0$ \textit{for any} $X,Y\in \Gamma (TM).$%
then $(M^{n},g,\nabla ,\nabla ^{\ast })$ is called holomorphic statistical
manifold.
\end{definition}

From Lemma \ref{FUNDAMENTAL FORM 2} and Corollary \ref{KAEHLER2} we have
following.

\begin{corollary}[\protect\cite{FURUHATA}]
\label{KAEHLER3}Let $(N^{2n},g,J,\nabla ,\nabla ^{\ast })$ be a Kaehler
statistical manifold. Then$\ $\ the following three equations are equivalent:

1) $\nabla \Omega =0,$

2) $N^{2n}$holomorphic statistical manifold,

3) $\nabla ^{\ast }\Omega =0.$
\end{corollary}

\section{Statistical Almost Contact Metric Manifolds}

\begin{definition}
Let ($M^{2n+1},g,\nabla ,\nabla ^{\ast })$ be a statistical manifold. If $%
M^{2n+1}$ $\ $is an almost contact metric manifold then $M^{2n+1}$ is called
almost contact metric statistical manifold.
\end{definition}

Using anti-symmetry property of $\phi $ and the equation (\ref{STAT1}) we
have

\begin{lemma}
\label{Conjugate fi}Let ($M^{2n+1},g,\nabla ,\nabla ^{\ast })$ be an almost
contact metric statistical manifold. Then the following equation 
\begin{equation}
g((\nabla _{X}\phi )Y,Z)=-g(Y,(\nabla _{X}^{\ast }\phi )Z)  \label{AA3}
\end{equation}%
holds for any $X,Y,Z\in \Gamma (TM).$
\end{lemma}

Using (\ref{K1}) and (\ref{K2}), we easily find the next result.

\begin{lemma}
\label{CONSTAT}Let ($M^{2n+1},g,\nabla ,\nabla ^{\ast })$ be an almost
contact statistical manifold. Then%
\begin{eqnarray}
(\nabla _{X}\phi )Y &=&(\nabla _{X}^{0}\phi )Y+(\mathcal{K}_{X}\phi )Y,
\label{AAB1} \\
(\nabla _{X}^{\ast }\phi )Y &=&(\nabla _{X}^{0}\phi )Y-(\mathcal{K}_{X}\phi
)Y.  \label{AAB2}
\end{eqnarray}%
for any $X,Y\in \Gamma (TM).$
\end{lemma}

\begin{lemma}
\label{FUNDAMENTAL FORM}For an almost contact statistical manifold we have%
\begin{equation}
(\nabla _{X}\Phi )(Y,Z)=g((\nabla _{X}\phi )Y,Z)-2g(\mathcal{K}_{X}\phi Y,Z),
\label{AA4a}
\end{equation}%
and%
\begin{equation}
(\nabla _{X}^{\ast }\Phi )(Y,Z)=g((\nabla _{X}^{\ast }\phi )Y,Z)+2g(\mathcal{%
K}_{X}\phi Y,Z)  \label{AA5}
\end{equation}%
for any $X,Y,Z\in \Gamma (TM).$
\end{lemma}

By Lemma \ref{CONSTAT} and the relations (\ref{K1}) and (\ref{K2}) we easily
prove the following corollary.

\begin{corollary}
For an almost contact metric statistical manifold we have 
\begin{equation}
(\nabla _{X}\Phi )(Y,Z)=(\nabla _{X}^{0}\Phi )(Y,Z)-g(\mathcal{K}_{X}\phi
Y+\phi \mathcal{K}_{X}Y,Z)  \label{BB1}
\end{equation}%
and%
\begin{equation}
(\nabla _{X}^{\ast }\Phi )(Y,Z)=(\nabla _{X}^{0}\Phi )(Y,Z)+g(\mathcal{K}%
_{X}\phi Y+\phi \mathcal{K}_{X}Y,Z)  \label{BB2}
\end{equation}%
for any $X,Y,Z\in \Gamma (TM).$
\end{corollary}

By (\ref{BB1}) and (\ref{BB2}) we have

\begin{corollary}[\protect\cite{FURUHATA2}]
For an almost contact metric statistical manifold we have 
\begin{equation}
(\nabla _{X}\Phi )(Y,Z)-(\nabla _{X}^{\ast }\Phi )(Y,Z)=-2g(\mathcal{K}%
_{X}\phi Y+\phi \mathcal{K}_{X}Y,Z)  \label{BBB1}
\end{equation}%
for any $X,Y,Z\in \Gamma (TM).$
\end{corollary}

For an almost contact metric manifold, the covariant derivative with respect
to Riemannian connection $\nabla ^{0}$ is given by%
\begin{eqnarray}
2g((\nabla ^{0}\phi )Y,Z) &=&3d\Phi (X,Y,Z)-3d\Phi (X,\phi Y,\phi
Z)+g(N^{(1)}(X,Y),\phi Z)  \notag \\
&&+((\mathcal{L}_{\phi X}\eta )(Y)-(\mathcal{L}_{\phi Y}\eta )(X))\eta (X)
\label{BB3} \\
&&+2d\eta (\phi Y,X)\eta (Z)-2d\eta (\phi Z,X)\eta (Y).  \notag
\end{eqnarray}%
(see \cite{Blair}).

Using (\ref{AAB1}), (\ref{AAB2}) and (\ref{BB3}), we have

\begin{theorem}
\label{FUNDAMENTAL FORM 4A}Let ($M^{2n+1},g,\phi ,\nabla ,\nabla ^{\ast })$
be an almost contact metric statistical manifold. The covariant derivatives $%
\nabla \phi ,\nabla ^{\ast }\phi $ of $J$ with respect to the torsion free
connections $\nabla $ and $\nabla ^{\ast }$ are given by%
\begin{eqnarray}
2g((\nabla _{X}\phi )Y,Z) &=&2g((\mathcal{K}_{X}\phi )Y,Z)  \notag \\
&&+3d\Phi (X,Y,Z)-3d\Phi (X,\phi Y,\phi Z)+g(N^{(1)}(X,Y),\phi Z)  \notag \\
&&+((\mathcal{L}_{\phi X}\eta )(Y)-(\mathcal{L}_{\phi Y}\eta )(X))\eta (X)
\label{BB4} \\
&&+2d\eta (\phi Y,X)\eta (Z)-2d\eta (\phi Z,X)\eta (Y),  \notag
\end{eqnarray}%
\begin{eqnarray}
2g((\nabla _{X}^{\ast }\phi )Y,Z) &=&-2g(\mathcal{K}_{X}\phi )Y,Z)  \notag \\
&&3d\Phi (X,Y,Z)-3d\Phi (X,\phi Y,\phi Z)+g(N^{(1)}(X,Y),\phi Z)  \notag \\
&&+((\mathcal{L}_{\phi X}\eta )(Y)-(\mathcal{L}_{\phi Y}\eta )(X))\eta (X)
\label{BB5} \\
&&+2d\eta (\phi Y,X)\eta (Z)-2d\eta (\phi Z,X)\eta (Y)  \notag
\end{eqnarray}%
for any $X,Y,Z\in \Gamma (TM).$
\end{theorem}

\section{Almost Cosymplectic Statistical Manifolds}

For an almost cosymplectic statistical manifold we define the $(1,1)$-tensor
fields $\mathcal{A}$ , $\mathcal{A}^{\ast }$ and $\mathcal{A}^{0}$ by%
\begin{equation}
\mathcal{A}X=-\nabla _{X}\xi ,\mathcal{A}^{\ast }X=-\nabla _{X}^{\ast }\xi 
\text{ and }\mathcal{A}^{0}X=-\nabla _{X}^{0}\xi ,\text{ }\forall X\in
\Gamma (TM).  \label{AA1}
\end{equation}%
Since $2\nabla ^{0}=\nabla +\nabla ^{\ast },$ by (\ref{AA1}), we obtain 
\begin{equation}
2\mathcal{A}^{0}=\mathcal{A}+\mathcal{A}^{\ast }.  \label{AA2}
\end{equation}

\begin{proposition}
\label{Proposition AFI}For an almost cosymplectic statistical manifold we
have%
\begin{eqnarray*}
i)\text{ }\mathcal{L}_{\xi }\eta  &=&0,\text{ }ii)\text{\ }g(\mathcal{A}%
X,Y)=g(X,\mathcal{A}Y), \\
\text{ }iii)\text{\ }g(\mathcal{A}^{\ast }X,Y) &=&g(X,\mathcal{A}^{\ast }Y)%
\text{, \ }iv)\mathcal{A\xi =-A}^{\ast }\xi =\emph{K}_{\xi }\xi , \\
\text{\ }v)\text{ }(\nabla _{\xi }\phi )X &=&\phi \mathcal{A}X+\mathcal{A}%
^{\ast }\phi X,\text{ \ }vi)(\nabla _{\xi }^{\ast }\phi )X=\phi \mathcal{A}%
^{\ast }X+\mathcal{A}\phi X, \\
vii)\,\mathcal{A}\phi +\phi \mathcal{A} &=&-(\mathcal{A}^{\ast }\phi +\phi 
\mathcal{A}^{\ast })
\end{eqnarray*}%
where $\mathcal{L}$ indicates the operator of the Lie differentiation, $X,Y$
are arbitrary vector fields on $M.$
\end{proposition}

\begin{proof}
Using Cartan magic formula 
\begin{equation*}
\mathcal{L}_{\xi }\eta =di_{\xi }(\eta )+i_{\xi }d(\eta ),
\end{equation*}%
and $d\eta =0,i_{\xi }\eta =1$ we obtain $i)$. Again noting $\eta $ is
closed and using (\ref{STAT1}) we have 
\begin{eqnarray*}
0 &=&2d\eta (X,Y)=X\eta (Y)-Y\eta (X)-\eta ([X,Y)) \\
&=&-g(Y,\mathcal{A}^{\ast }X)+g(X,\mathcal{A}^{\ast }Y)
\end{eqnarray*}%
and by help of similar calculations respect to torsion free affine
connection $\nabla $, we obtain%
\begin{equation*}
g(\mathcal{A}X,Y)=g(X,\mathcal{A}Y).
\end{equation*}%
So we get $ii)$ and $iii).$

From (\ref{K1}) and (\ref{K2}) we get $iv).$

Cartan magic formula 
\begin{equation}
\mathcal{L}_{V}\Omega =di_{V}(\Omega )+i_{V}d(\Omega )  \label{CMF}
\end{equation}%
is valid for any form $\Omega \in \wedge (M)$ and $V$ $\in \Gamma (TM).$ Let
us apply (\ref{CMF}) to the (\ref{3}). Since $i_{\xi }(\Phi )=0$ and $d\Phi
=0$, we have 
\begin{equation}
\mathcal{L}_{\xi }\Phi =0.  \label{LD2F}
\end{equation}%
On the other hand, the Lie derivative of $2$-form $\Phi $ with respect to
characteristic vector field $\xi $ can be expressed as%
\begin{eqnarray*}
(\mathcal{L}_{\xi }\Phi )(X,Y) &=&\mathcal{L}_{\xi }\Phi (X,Y)-\Phi (%
\mathcal{L}_{\xi }X,Y)-\Phi (X,\mathcal{L}_{\xi }Y) \\
&=&\xi g(\phi X,Y)-g(\phi \nabla _{\xi }X,Y)+g(\phi \nabla _{X}\xi ,Y) \\
&&-g(\phi X,\nabla _{\xi }Y)+g(\phi X,\nabla _{Y}\xi ) \\
&=&g(\nabla _{\xi }\phi X,Y)+g(\phi X,\nabla _{\xi }^{\ast }Y) \\
&&-g(\phi \nabla _{\xi }X,Y)+g(\phi \nabla _{X}\xi ,Y) \\
&&-g(\phi X,\nabla _{\xi }Y)+g(\phi X,\nabla _{Y}\xi ) \\
&=&g((\nabla _{\xi }\phi )X,Y)-g(\phi \mathcal{A}X,Y)-g(\phi X,\mathcal{A}%
Y)+g(\phi X,(\mathcal{A}-\mathcal{A}^{\ast })Y) \\
&=&g((\nabla _{\xi }\phi )X,Y)-g(\phi \mathcal{A}X,Y)-g(\mathcal{A}^{\ast
}\phi X,Y).
\end{eqnarray*}%
We thus obtain $(\nabla _{\xi }\phi )X=\phi \mathcal{A}X+\mathcal{A}^{\ast
}\phi X,$ According to conjugate connection $\nabla ^{\ast }$ we conclude
that $(\nabla _{\xi }^{\ast }\phi )X=\phi \mathcal{A}^{\ast }X+\mathcal{A}%
\phi X.$ So we have $v)$ and $vi).$ We know that $\mathcal{A}^{0}\phi +\phi 
\mathcal{A}^{0}=0$ is valid for almost cosymplectic manifold. Thus, by (\ref%
{AA2}) we get $vii).$
\end{proof}

\begin{remark}
\label{AKSI}By Proposition \ref{Proposition AFI} , we say that $\mathcal{A}%
\xi =0$ if and only if $\mathcal{A}^{\ast }\xi =0$ for an almost
cosymplectic statistical manifold.
\end{remark}

\begin{proposition}
Let ($M^{2n+1},g,\nabla ,\nabla ^{\ast })$ be an almost cosymplectic
statistical manifold. Then%
\begin{eqnarray}
(\nabla _{X}\Phi )(Y,Z)+(\nabla _{Z}\Phi )(X,Y)+(\nabla _{Y}\Phi )(Z,X) &=&0,
\label{KF1a} \\
(\nabla _{X}^{\ast }\Phi )(Y,Z)+(\nabla _{Z}^{\ast }\Phi )(X,Y)+(\nabla
_{Y}^{\ast }\Phi )(Z,X) &=&0  \label{KF2a}
\end{eqnarray}%
for any $X,Y,Z\in \Gamma (TM)$ .

\begin{proof}
Since $M^{2n+1}$ is almost cosymplectic the relation 
\begin{equation}
(\nabla _{X}^{0}\Phi )(Y,Z)+(\nabla _{Z}^{0}\Phi )(X,Y)+(\nabla _{Y}^{0}\Phi
)(Z,X)=0  \label{KF3a}
\end{equation}%
holds. If we insert the relation (\ref{BB1}) into above expression, we find%
\begin{eqnarray}
0 &=&(\nabla _{X}\Phi )(Y,Z)+(\nabla _{Y}\Phi )(Z,X)+(\nabla _{Z}\Phi )(X,Y)
\notag \\
&&+g(\mathcal{K}_{X}\phi Y+\phi \mathcal{K}_{X}Y,Z)+g(\mathcal{K}_{Y}\phi
Z+\phi \mathcal{K}_{X}Z,X)  \label{KF13c} \\
&&+g(\mathcal{K}_{Z}\phi X+\phi \mathcal{K}_{Z}X,Y).  \notag
\end{eqnarray}%
If we make use of the relation (\ref{AZIZ IDENTITIY}), we have%
\begin{equation}
0=(\nabla _{X}\Phi )(Y,Z)+(\nabla _{Y}\Phi )(Z,X)+(\nabla _{Z}\Phi )(X,Y).
\label{KF13e}
\end{equation}

Employing (\ref{BB2}) into (\ref{KF3a}) and using (\ref{KF13e}), we can
easily verify that the equality 
\begin{equation*}
0=(\nabla _{X}^{\ast }\Phi )(Y,Z)+(\nabla _{Z}^{\ast }\Phi )(X,Y)+(\nabla
_{Y}^{\ast }\Phi )(Z,X)
\end{equation*}%
holds.
\end{proof}
\end{proposition}

\begin{proposition}
\label{LKSI}\ \ For an almost cosymplectic statistical manifold we have\ \ \
\ \ \ 
\begin{eqnarray*}
i)(\mathcal{L}_{\xi }g)(X,Y) &=&-2g(\mathcal{A}^{0}X,Y)=-g((\mathcal{A+A}%
^{\ast })X,Y), \\
ii)\text{ }(\nabla _{X}\eta )Y &=&(\nabla _{X}\eta )Y,\text{ }iii)\text{ }%
(\nabla _{X}^{\ast }\eta )Y=(\nabla _{X}^{\ast }\eta )Y.
\end{eqnarray*}%
\ \ \ 
\end{proposition}

\begin{proof}
By direct calculations and using (\ref{STAT1}), we find 
\begin{eqnarray*}
(\mathcal{L}_{\xi }g)(X,Y) &=&\xi g(X,Y)-g(\left[ \xi ,X\right] ,Y)-g(X,%
\left[ \xi ,Y\right] ) \\
&=&g(\nabla _{\xi }X,Y)+g(X,\nabla _{\xi }^{\ast }Y) \\
&&-g(\nabla _{\xi }X,Y)-g(\mathcal{A}X,Y) \\
&&-g(X,\nabla _{\xi }Y)-g(X,\mathcal{A}Y).
\end{eqnarray*}

Due to the symmetry of the operator $A$ , we get 
\begin{equation*}
(\mathcal{L}_{\xi }g)(X,Y)=g(X,\nabla _{\xi }^{\ast }Y-\nabla _{\xi }Y)-2g(%
\mathcal{A}X,Y).
\end{equation*}%
In view of (\ref{K2}), we obtain 
\begin{eqnarray*}
(\mathcal{L}_{\xi }g)(X,Y) &=&-2g(X,\mathcal{K}_{\xi }Y)-2g(\mathcal{A}X,Y)
\\
&=&-2g(X,\mathcal{K}_{Y}\xi )-2g(\mathcal{A}X,Y) \\
&=&-2g(X,\mathcal{A}^{0}Y).
\end{eqnarray*}%
It is clear that one has%
\begin{equation*}
(\nabla _{X}\eta )Y=-g(Y,\mathcal{A}^{\ast }X)=-g(\mathcal{A}^{\ast
}Y,X)=(\nabla _{Y}\eta )X
\end{equation*}%
which completes $ii).$

Since $d\eta =0,$ one can easily get 
\begin{equation}
(\nabla _{X}^{0}\eta )Y=(\nabla _{X}^{0}\eta )Y.  \label{C1}
\end{equation}%
From $ii)$ and (\ref{C1}), we find that $(\nabla _{X}^{\ast }\eta )Y=(\nabla
_{X}^{\ast }\eta )Y.$
\end{proof}

\begin{proposition}
\label{SECOND FUNDFORM}For an almost cosymplectic statistical manifold we
have%
\begin{eqnarray}
\text{ }(\nabla _{X}\Phi )(Y,\phi Z)+(\nabla _{X}^{\ast }\Phi )(Z,\phi Y)
&=&\eta (Y)g(\mathcal{A}X,Z)+\eta (Z)g(\mathcal{A}^{\ast }X,Y)  \label{DF1}
\\
(\nabla _{X}^{\ast }\Phi )(\phi Z,\phi Y)-(\nabla _{X}\Phi )(Y,Z) &=&\eta
(Y)g(\mathcal{A}X,\phi Z)-\eta (Z)g(\mathcal{A}X,\phi Y)  \label{DF2}
\end{eqnarray}%
for any $X,Y,Z\in \Gamma (TM).$
\end{proposition}

\begin{proof}
Differentiating the identity $\phi ^{2}=-I+\eta \otimes \xi $ covariantly,
we obtain%
\begin{equation*}
(\nabla _{X}\phi )\phi Y+\phi (\nabla _{X}\phi )Y=-g(Y,\mathcal{A}^{\ast
}X)\xi -\eta (Y)\mathcal{A}X.
\end{equation*}%
Projecting this equality onto $Z$ and then using antisymmetry of $\phi ,$ we
get%
\begin{equation*}
g((\nabla _{X}\phi )\phi Y,Z)-g((\nabla _{X}\phi )Y,\phi Z)=-\eta (Z)g(%
\mathcal{A}^{\ast }X,Y)-\eta (Y)g(\mathcal{A}X,Z).
\end{equation*}%
Recalling $g((\nabla _{X}\phi )Y,Z)=-g(Y,(\nabla _{X}^{\ast }\phi )Z),$ we
obtain 
\begin{equation*}
g(\phi Y,(\nabla _{X}^{\ast }\phi )Z)+g((\nabla _{X}\phi )Y,\phi Z)=\eta
(Z)g(\mathcal{A}^{\ast }X,Y)+\eta (Y)g(\mathcal{A}X,Z).
\end{equation*}%
Finally, by Lemma \ref{FUNDAMENTAL FORM}, we find (\ref{DF1}).

It is easily verified that 
\begin{equation*}
(\nabla _{X}^{\ast }\phi )\xi =\phi \mathcal{A}^{\ast }X
\end{equation*}%
for any $X\in \Gamma (TM).$ So we get%
\begin{eqnarray*}
(\nabla _{X}\Phi )(Y,\xi ) &=&g((\nabla _{X}\phi )Y,\xi )-2g(\mathcal{K}%
_{X}\phi Y,\xi ) \\
&=&-g(Y,(\nabla _{X}^{\ast }\phi )\xi )-2g(\phi Y,\mathcal{K}_{X}\xi ) \\
&=&-g(\phi \mathcal{A}^{\ast }X,Y)+2g(\mathcal{A}X,\phi Y)-2g(\mathcal{A}%
^{0}X,\phi Y) \\
&=&-g(\phi \mathcal{A}^{\ast }X,Y)+2g(\mathcal{A}X,\phi Y)-g(\mathcal{A}%
X,\phi Y)-g(\mathcal{A}^{\ast }X,\phi Y) \\
&=&g(\mathcal{A}X,\phi Y).
\end{eqnarray*}%
Replacing $Z$ by $\phi Z$ in (\ref{DF1}) and using foregoing equality, we
finally arrive at the (\ref{DF2}).
\end{proof}

Since an almost cosymplectic manifold is characterized by $\nabla ^{0}\phi
=0 $, by Corollary \ref{CONSTAT}\ \ we have

\begin{theorem}
Let ($M^{2n+1},g,\phi ,\nabla ,\nabla ^{\ast })$ be a cosymplectic
statistical manifold. Then%
\begin{equation}
(\nabla _{X}\phi )Y=(\mathcal{K}_{X}\phi )Y  \label{DAZIZ1}
\end{equation}%
\begin{equation}
(\nabla _{X}^{\ast }\phi )Y=-(\mathcal{K}_{X}\phi )Y  \label{DAZIZ2}
\end{equation}%
for any $X,Y\in \Gamma (TM).$
\end{theorem}

\begin{theorem}
Let ($M^{2n+1},g,\phi ,\nabla ,\nabla ^{\ast })$ be an almost contact
statistical manifold. Then ($M^{2n+1},g,\phi ,\nabla ,\nabla ^{\ast })$ be a
cosymplectic statistical manifold if and only if%
\begin{equation}
\nabla _{X}\phi Y-\phi \nabla _{X}^{\ast }Y=\mathcal{K}_{X}\phi Y+\phi 
\mathcal{K}_{X}Y  \label{DAZIZ3}
\end{equation}%
for any $X,Y\in \Gamma (TM).$
\end{theorem}

\begin{corollary}
($M^{2n+1},g,\phi ,\nabla ,\nabla ^{\ast })$ \ is a cosymplectic \
holomorphic statistical manifold if and only if $\ \nabla _{X}\phi Y=\phi
\nabla _{X}^{\ast }Y$ for any $X,Y\in \Gamma (TM).$
\end{corollary}

Let $t$ be the coordinate on $%
\mathbb{R}
.$ We denote by $\partial _{t}=\frac{\partial }{\partial _{t}}$ the unit
vector fileld on $%
\mathbb{R}
.$ Define affine connections $\nabla $ and $\nabla ^{\ast }$ on $%
\mathbb{R}
$ by 
\begin{equation}
^{%
\mathbb{R}
}\nabla _{\partial _{t}}\partial _{t}=\lambda (t)\partial _{t}\text{ \ and \ 
}^{%
\mathbb{R}
}\nabla _{\partial _{t}}^{\ast }\partial _{t}=\lambda ^{\ast }(t)\partial
_{t}=-\lambda (t)\partial _{t}\text{ ,}  \label{STAT2}
\end{equation}%
where $\lambda :%
\mathbb{R}
\rightarrow 
\mathbb{R}
$ is a smooth function. It is clear that $(g_{%
\mathbb{R}
}=dt^{2},^{%
\mathbb{R}
}\nabla ,^{%
\mathbb{R}
}\nabla ^{\ast })$ is a dualistic structure on $%
\mathbb{R}
.$

By \cite{TO}, the following proposition can be given.

\begin{proposition}[ \protect\cite{TO}]
\label{TO1} Let $(g_{N},^{N}\nabla ,^{N}\nabla ^{\ast })$ be dualistic
structures on $N$. \ Let us consider $(M=%
\mathbb{R}
\times N,<,>=$ $dt^{2}+g_{N})$ Riemannian product manifold . If $U,V$ $\ $%
are vector fields on $N$ and $\bar{\nabla},\bar{\nabla}^{\ast }$ satisfy
following relations on $%
\mathbb{R}
\times N:$

(a) $\bar{\nabla}_{\partial _{t}}\partial _{t}=$ $\lambda (t)\partial _{t}$

(b) $\bar{\nabla}_{_{\partial _{t}}}U=\bar{\nabla}_{U}\partial _{t}=0$

(c) $\bar{\nabla}_{U}V$ $=$ $\ ^{N}\nabla _{U}V$

and

(i) $\bar{\nabla}_{\partial _{t}}^{\ast }\partial _{t}=$ $-\lambda
(t)\partial _{t}$

(ii) $\bar{\nabla}_{\partial _{t}}^{\ast }U=\bar{\nabla}_{U}^{\ast }\partial
_{t}=0$

(iii) $\bar{\nabla}_{U}^{\ast }V$ $=$ $\ ^{N}\nabla _{U}^{\ast }V$

then $(<,>,\bar{\nabla},\bar{\nabla}^{\ast })$ is a dualistic structure on $%
M=%
\mathbb{R}
\times N.$
\end{proposition}

It is well known that a cosymplectic manifold is a locally product of an
open interval and a Kahlerian manifold \cite{DacOl1}. So we can give

\begin{theorem}
\label{YAZLA16}Let $(N,g,\nabla ,\nabla ^{\ast },J)$ be a Kaehler
statistical manifold and $(%
\mathbb{R}
,\nabla ^{%
\mathbb{R}
},\nabla ^{\ast 
\mathbb{R}
}dt)$ be statistical manifold. Under the Proposition \ref{TO1}, $%
\mathbb{R}
\times N$ is a cosymplectic statistical manifold.
\end{theorem}

\begin{example}[ \protect\cite{DACKO}]
\label{DACKO1}If \ a group operation in $%
\mathbb{R}
^{3}$ is defined as 
\begin{equation*}
(t,x,y)\ast (s,u,v)=(t+s,e^{-t}u+x,e^{t}v+y)
\end{equation*}%
for any $(t,x,y),(s,u,v)\in $ $%
\mathbb{R}
^{3}$ then $(%
\mathbb{R}
^{3},\ast )$ is a Lie group which is called the solvable non-nilpotent Lie
group. The following set of left-invariant vector fields forms an
orthonormal basis for the corresponding Lie algebra:%
\begin{equation*}
E_{0}=\frac{\partial }{\partial t},E_{1}=e^{-t}\frac{\partial }{\partial x}%
,E_{2}=e^{t}\frac{\partial }{\partial y}.
\end{equation*}%
According to this base , one can construct almost contact metric structure $%
(\phi ,\xi ,\eta ,g)$ on $%
\mathbb{R}
^{3}$ as follows:%
\begin{eqnarray}
\eta  &=&dt,\xi =E_{0},\phi =e^{2t}dx\otimes \frac{\partial }{\partial y}%
-e^{-2t}dy\otimes \frac{\partial }{\partial x},  \label{SEFADIYE1} \\
\text{ \ \ \ }g &=&dt\otimes dt+e^{2t}dx\otimes dx+e^{-2t}dy\otimes dy. 
\notag
\end{eqnarray}%
We obviously get $(%
\mathbb{R}
^{3},\phi ,\xi ,\eta ,g)$ is an almost cosymplectic manifold.
\end{example}

With respect to Example \ref{DACKO1}, we provide an example on almost
cosymplectic statistical manifold on $%
\mathbb{R}
^{3}.$

\begin{example}
\label{YAZLA}Consider Example \ref{DACKO1} for almost cosymplectic
statistical manifolds . By (\ref{SEFADIYE1}) and Koszula formula, we can now
proceed to calculate the Levi-Civita connections 
\begin{equation}
\begin{array}{lll}
\nabla _{E_{1}}^{0}E_{1}=-E_{0}, & \nabla _{E_{2}}^{0}E_{1}=0, & \nabla
_{E_{0}}^{0}E_{1}=0, \\ 
\nabla _{E_{1}}^{0}E_{2}=0, & \nabla _{E_{2}}^{0}E_{2}=E_{0}, & \nabla
_{E_{0}}^{0}E_{2}=0, \\ 
\nabla _{E_{1}}^{0}E_{0}=E_{1}, & \nabla _{E_{2}}^{0}E_{0}=-E_{2}, & \nabla
_{E_{0}}^{0}E_{0}=0.%
\end{array}
\label{OZCAN1}
\end{equation}%
Now we define torsion-free affine connections $\nabla ,\nabla ^{\ast }$ as
follows 
\begin{equation}
\begin{array}{lll}
\nabla _{E_{1}}E_{1}=-E_{0}+E_{2}, & \nabla _{E_{2}}E_{1}=E_{1}+E_{0}, & 
\nabla _{E_{0}}E_{1}=E_{2}, \\ 
\nabla _{E_{1}}E_{2}=E_{1}+E_{0}, & \nabla _{E_{2}}E_{2}=E_{0}+E_{2}, & 
\nabla _{E_{0}}E_{2}=E_{1}, \\ 
\nabla _{E_{1}}E_{0}=E_{1}+E_{2}, & \nabla _{E_{2}}E_{0}=-E_{2}+E_{1}, & 
\nabla _{E_{0}}E_{0}=E_{0}.%
\end{array}
\label{ONUR1}
\end{equation}%
$\ \ $
\end{example}

$\ \ \ \ \ \ \ \ \ \ \ \ \ \ \ \ \ \ \ \ \ \ \ \ \ \ \ \ \ \ \ \ \ \ \ \ \ \
\ \ \ \ \ \ \ \ \ \ \ \ \ \ \ \ \ \ \ \ \ \ \ \ \ \ \ \ \ \ \ \ \ \ \ \ \ \
\ \ \ \ \ \ \ \ \ \ \ \ \ \ \ \ \ \ \ \ \ \ \ \ \ \ \ \ \ \ \ \ \ \ \ \ \ \
\ \ \ \ \ \ \ \ \ \ \ $%
\begin{equation}
\begin{array}{lll}
\nabla _{E_{1}}^{\ast }E_{1}=-E_{0}-E_{2}, & \nabla _{E_{2}}^{\ast
}E_{1}=-E_{1}-E_{0}, & \nabla _{E_{0}}^{\ast }E_{1}=-E_{2}, \\ 
\nabla _{E_{1}}^{\ast }E_{2}=-E_{1}-E_{0}, & \nabla _{E_{2}}^{\ast
}E_{2}=E_{0}-E_{2}, & \nabla _{E_{0}}^{\ast }E_{2}=-E_{1}, \\ 
\nabla _{E_{1}}^{\ast }E_{0}=E_{1}-E_{2}, & \nabla _{E_{2}}^{\ast
}E_{0}=-E_{2}-E_{1}, & \nabla _{E_{0}}^{\ast }E_{0}=-E_{0}.%
\end{array}
\label{ONUR2}
\end{equation}%
\textit{where} 
\begin{equation}
\begin{array}{lll}
\mathcal{K}_{E_{1}}E_{1}=E_{2}, & \mathcal{K}_{E_{2}}E_{1}=E_{1}+E_{0} & 
\mathcal{K}_{E_{0}}E_{1}=E_{2}, \\ 
\mathcal{K}_{E_{1}}E_{2}=E_{1}+E_{0}, & \mathcal{K}_{E_{2}}E_{2}=E_{2}, & 
\mathcal{K}_{E_{0}}E_{2}=E_{1}, \\ 
\mathcal{K}_{E_{1}}E_{0}=E_{2}, & \mathcal{K}_{E_{2}}E_{0}=E_{1}, & \mathcal{%
K}_{E_{0}}E_{0}=E_{0}.%
\end{array}
\label{ONUR22}
\end{equation}%
\textit{Hence we have }

\begin{equation*}
Zg(X,Y)=g(\nabla _{Z}X,Y)+g(X,\nabla _{Z}^{\ast }Y)
\end{equation*}%
\textit{for any }$X,Y,Z\in \Gamma (TM).$\textit{\ It means that }$(%
\mathbb{R}
^{3},g,\nabla ,\nabla ^{\ast },\phi )$\textit{\ is an almost cosymplectic
statistical manifold. We notice that }$\mathcal{A}E_{0}$ and $\mathcal{A}%
^{\ast }E_{0}$ \textit{are different from zero.}

\textit{If torsion-free affine connections} $\nabla ,\nabla ^{\ast }$ 
\textit{satisfy the followings}

\begin{equation*}
\begin{array}{lll}
\nabla _{E_{1}}E_{1}=-E_{0}+E_{2}, & \nabla _{E_{2}}E_{1}=E_{1}, & \nabla
_{E_{0}}E_{1}=0, \\ 
\nabla _{E_{1}}E_{2}=E_{1}, & \nabla _{E_{2}}E_{2}=E_{0}+E_{2}, & \nabla
_{E_{0}}E_{2}=0, \\ 
\nabla _{E_{1}}E_{0}=E_{1}, & \nabla _{E_{2}}E_{0}=-E_{2}, & \nabla
_{E_{0}}E_{0}=0.%
\end{array}%
\end{equation*}%
$\ \ $

$\ \ \ \ \ \ \ \ \ \ \ \ \ \ \ \ \ \ \ \ \ \ \ \ \ \ \ \ \ \ \ \ \ \ \ \ \ \
\ \ \ \ \ \ \ \ \ \ \ \ \ \ \ \ \ \ \ \ \ \ \ \ \ \ \ \ \ \ \ \ \ \ \ \ \ \
\ \ \ \ \ \ \ \ \ \ \ \ \ \ \ \ \ \ \ \ \ \ \ \ \ \ \ \ \ \ \ \ \ \ \ \ \ \
\ \ \ \ \ \ \ \ \ \ \ $%
\begin{equation*}
\begin{array}{lll}
\nabla _{E_{1}}^{\ast }E_{1}=-E_{0}-E_{2}, & \nabla _{E_{2}}^{\ast
}E_{1}=-E_{1}, & \nabla _{E_{0}}^{\ast }E_{1}=0, \\ 
\nabla _{E_{1}}^{\ast }E_{2}=-E_{1}, & \nabla _{E_{2}}^{\ast
}E_{2}=E_{0}-E_{2}, & \nabla _{E_{0}}^{\ast }E_{2}=0, \\ 
\nabla _{E_{1}}^{\ast }E_{0}=E_{1}, & \nabla _{E_{2}}^{\ast }E_{0}=-E_{2}, & 
\nabla _{E_{0}}^{\ast }E_{0}=0.%
\end{array}%
\end{equation*}%
\textit{then} $(%
\mathbb{R}
^{3},g,\nabla ,\nabla ^{\ast },\phi )$\textit{\ is again an almost
cosymplectic statistical manifold. In this case }$\mathcal{A}=$ $\mathcal{A}%
^{\ast }=$ $\mathcal{A}^{0}$ \textit{and \ also the integral curves of} $%
E_{0}=\xi $\textit{\ are geodesics respect to affine connections} $\nabla $,$%
\nabla ^{\ast }$.

$\ \ \ \ \ \ \ \ \ \ \ \ \ \ \ \ \ \ \ \ \ \ \ \ \ \ \ \ \ \ \ \ \ \ \ \ \ \
\ \ \ \ \ \ \ \ \ \ \ \ \ \ \ \ \ \ \ \ \ \ \ \ \ \ \ \ \ \ \ \ \ \ \ \ \ \
\ \ \ \ \ \ \ \ \ \ \ \ \ \ \ \ \ \ \ \ \ \ \ \ \ \ \ \ \ \ \ \ \ \ \ \ \ \
\ \ \ \ \ \ \ \ \ \ \ \ \ \ \ \ \ \ \ \ \ \ \ \ \ \ \ \ \ \ \ \ \ \ \ \ \ \
\ \ \ \ \ \ \ \ \ \ \ \ \ \ \ \ \ \ \ \ \ \ \ \ \ \ \ \ \ \ \ \ \ \ \ \ \ \
\ \ \ \ \ \ \ \ \ \ \ \ \ \ \ \ \ \ \ \ \ \ \ \ \ \ \ \ \ \ \ \ \ \ \ \ \ \
\ \ \ \ \ \ \ \ \ \ \ \ \ \ \ \ \ \ \ \ \ \ \ \ \ \ \ \ \ \ \ \ \ \ \ $

Now we will give an application for Equations (\ref{AAB1}) and (\ref{AAB2}).
Let $M^{2n+1}=(M,\phi ,\xi ,\eta ,g)$ be an almost cosymplectic statistical
manifold. By the definition, the form $\eta $ is closed, therefore
distribution $\mathcal{D}:\eta =0$ is completely integrable. Each leaf of
the foliation, determined by $\mathcal{D}$, carries an almost Kaehler
structure $(J,<,>)$%
\begin{equation*}
J\bar{X}=\phi \bar{X},\text{ \ \ \ \ }\left\langle \bar{X},\bar{Y}%
\right\rangle =g(\bar{X},\bar{Y}),
\end{equation*}%
$\bar{X},\bar{Y}$ are vector fields tangent to the leaf. If this structure
is Kaehler statistical, leaf is called a Kaehler statistical leaf. Other
hand Dacko and Olszak \cite{DacOl1},\cite{OlK} announced that an almost
cosymplectic manifold has Kaehler leaves if and only if 
\begin{equation*}
(\nabla _{X}^{0}\phi )Y=g(\mathcal{A}^{0}X,\phi Y)\xi +\eta (Y)\phi \mathcal{%
A}^{0}X,\text{ \ \ \ }\mathcal{A}^{0}=-\nabla ^{0}\xi .
\end{equation*}%
Using the last equation in (\ref{AAB1}) and (\ref{AAB2}), we have following.

\begin{theorem}
An almost cosymplectic statistical manifold has Kaehler statistical leaves
if and only if 
\begin{eqnarray*}
(\nabla _{X}\phi )Y &=&(\mathcal{K}_{X}\phi )Y+g(\mathcal{A}^{0}X,\phi Y)\xi
+\eta (Y)\phi \mathcal{A}^{0}X, \\
(\nabla _{X}^{\ast }\phi )Y &=&-(\mathcal{K}_{X}\phi )Y+g(\mathcal{A}%
^{0}X,\phi Y)\xi +\eta (Y)\phi \mathcal{A}^{0}X.
\end{eqnarray*}
\end{theorem}

The (1,1) tensor field $\mathcal{K}\circ \phi +\phi \circ \mathcal{K}$ is
important to define hololmorphic statistical manifolds. So above theorem can
be given following.

\begin{corollary}
An almost cosymplectic statistical manifold has\ Kaehler statistical leaves
if and only if%
\begin{equation}
\nabla _{X}\phi Y-\phi \nabla _{X}^{\ast }Y=g(\mathcal{A}^{0}X,\phi Y)\xi
+\eta (Y)\phi \mathcal{A}^{0}X+\mathcal{K}_{X}\phi Y+\phi \mathcal{K}_{X}Y
\label{O1}
\end{equation}%
for any $X,Y\in \Gamma (TM)$ .
\end{corollary}

\section{Curvature properties on almost cosymplectic statistical manifolds}

In this section, we study curvature properties of an almost cosymplectic
statistical manifold\textit{.} By simple computations, we have the following
theorem.

\begin{theorem}
\label{b30} Let ($M^{2n+1},\phi ,\xi ,\eta ,g)$ be an \textit{almost
cosymplectic statistical manifold.} Then, for any $X,Y\in \Gamma
(TM^{2n+1}), $%
\begin{eqnarray}
R(X,Y)\xi &=&(\nabla _{X}\mathcal{A})Y-(\nabla _{Y}\mathcal{A})X,  \label{R0}
\\
R^{\ast }(X,Y)\xi &=&(\nabla _{X}^{\ast }\mathcal{A}^{\ast })Y-(\nabla
_{Y}^{\ast }\mathcal{A}^{\ast })X.  \label{R00}
\end{eqnarray}
\end{theorem}

We define $(1,1)$ tensor fields\ $h^{0},h$ and \ $h^{\ast }$ on $M^{2n+1}$
by 
\begin{eqnarray}
h^{0} &=&\frac{1}{2}(\mathcal{L}_{\xi }\phi )  \label{R01} \\
h &=&\frac{1}{2}(\mathcal{A}\phi -\phi \mathcal{A}),h^{\ast }=\frac{1}{2}(%
\mathcal{A}^{\ast }\phi -\phi \mathcal{A}^{\ast }),  \label{R02}
\end{eqnarray}%
respevtively. It is proved that $h^{0}$ is symmetric and $h^{0}=$ $\mathcal{A%
}^{0}\phi $ in \cite{DacOl1}. In the following proposition we establish some
properties of the tensor fields $h$ and $h^{\ast }$ and provide a relation
between $h^{0},h$ and $h^{\ast }.$

\begin{proposition}
Let ($M^{2n+1},\phi ,\xi ,\eta ,g)$ be an \textit{almost cosymplectic
statistical manifold.} Then 
\begin{equation}
g(hX,Y)=g(X,hY)\text{ and }g(h^{\ast }X,Y)=g(X,h^{\ast }Y),  \label{R03}
\end{equation}%
\begin{equation}
h^{0}X=hX+\frac{1}{2}(\mathcal{K}_{\zeta }\phi )X\text{ and }h^{0}X=h^{\ast
}X-\frac{1}{2}(\mathcal{K}_{\zeta }\phi )X.  \label{R04}
\end{equation}
\end{proposition}

\begin{proof}
Using the antisymmetry of $\phi $ and the symmetry of $A,A^{\ast },$ we have
the equation (\ref{R03}). By direct computations we obtain 
\begin{eqnarray*}
(\mathcal{L}_{\zeta }\phi )X &=&(\nabla _{\zeta }\phi )X+\mathcal{A}\phi
X-\phi \mathcal{A}X \\
&=&(\nabla _{\zeta }\phi )X+2hX \\
&&\overset{(\ref{AAB1})}{=}(\mathcal{K}_{\xi }\phi )X+2hX
\end{eqnarray*}%
and%
\begin{eqnarray*}
(\mathcal{L}_{\zeta }\varphi )X &=&(\nabla _{\zeta }^{\ast }\varphi
)X+2h^{\ast }X \\
&&\overset{\ref{AAB2}}{=}-(\mathcal{K}_{\xi }\phi )X+2h^{\ast }X.
\end{eqnarray*}%
So from above the last equations we get%
\begin{equation}
h^{\ast }X-hX=(\mathcal{K}_{\xi }\phi )X.  \label{R05}
\end{equation}%
On the other hand one can easily obtain that%
\begin{equation}
h^{\ast }X+hX=2h^{0}X.  \label{R06}
\end{equation}%
By (\ref{R05}) and (\ref{R06}) we are led to (\ref{R04}).
\end{proof}

\begin{corollary}
\label{KLM}Let ($M^{2n+1},\phi ,\xi ,\eta ,g)$ be an \textit{almost
cosymplectic statistical manifold.} Then $\mathcal{K}_{\xi }\varphi =0$ if
and only if $\mathcal{A\phi }=-\phi \mathcal{A}^{\ast }$ and $\mathcal{A}%
^{\ast }\phi =-\phi \mathcal{A}$ if and only if $\ \nabla _{\xi }\phi
=0=\nabla _{\xi }^{\ast }\phi $
\end{corollary}

By (\ref{AA2}), (\ref{R0}) and (\ref{R00}) we find following.

\begin{theorem}
\label{R}Let ($M^{2n+1},\phi ,\xi ,\eta ,g)$ be an be an \textit{almost
cosymplectic statistical manifold.} Then, for any $X,Y\in $ $\Gamma (TM),$%
\begin{eqnarray}
4R^{0}(X,Y)\xi &=&R(X,Y)\xi +R^{\ast }(X,Y)\xi  \label{b3} \\
&&+(\nabla _{Y}^{\ast }A)X-\,(\nabla _{X}^{\ast }A)Y+(\nabla _{Y}A^{\ast
})X-\,(\nabla _{X}A^{\ast })Y.  \notag
\end{eqnarray}
\end{theorem}

For an almost cosymplectic manifold, \ we can give well known equality (see 
\cite{HAKAN}): 
\begin{equation}
R^{0}(X,\xi )\xi -\phi R^{0}(\phi X,\xi )\xi =-2(\mathcal{A}^{0})^{2}.
\label{b4}
\end{equation}

\begin{theorem}
\label{R2}Let ($M^{2n+1},\phi ,\xi ,\eta ,g)$ be an \textit{almost
cosymplectic statistical manifold.} We assume that $\mathcal{K}_{\xi
}\varphi =0$ and $\mathcal{A}\xi =0.$Then, for any $X\in $ $\Gamma (TM),$ we
have%
\begin{equation}
R(X,\xi )\xi -\phi R(\phi X,\xi )\xi +R^{\ast }(X,\xi )\xi -\phi R^{\ast
}(\phi X,\xi )\xi )=-2(\mathcal{A}^{2}+(\mathcal{A}^{\ast })^{2})X,
\label{RZETAZETA}
\end{equation}%
\begin{equation}
S(\xi ,\xi )+S^{\ast }(\xi ,\xi )=-tr(A^{2}+(A^{\ast })^{2}).
\label{SZETAZETA}
\end{equation}
\end{theorem}

\begin{proof}
If we replace $Y$ by $\xi $ in (\ref{b3}) and recall Remark \ref{AKSI} we
have 
\begin{eqnarray}
4R^{0}(X,\xi )\xi  &=&R(X,\xi )\xi +R^{\ast }(X,\xi )\xi   \notag \\
&&+(\nabla _{\xi }^{\ast }A)X+\,(\nabla _{\xi }A^{\ast })X  \notag \\
&&+\mathcal{A}\nabla _{X}^{\ast }\xi +\mathcal{A}^{\ast }\nabla _{X}\xi  
\notag \\
&=&R(X,\xi )\xi +R^{\ast }(X,\xi )\xi   \label{CM1} \\
&&+(\nabla _{\xi }^{\ast }A)X+\,(\nabla _{\xi }A^{\ast })X  \notag \\
&&-(AA^{\ast }+A^{\ast }A)X.  \notag
\end{eqnarray}%
Replacing $X$ by $\phi X$ in (\ref{CM1}) and then applying the tensor field $%
\phi $ both sides of the obtained equation and recalling that $\mathcal{A}%
\xi =0\overset{\text{Remark}\ref{AKSI}}{=}\mathcal{A}^{\ast }\xi $ we
readily find%
\begin{eqnarray}
4\phi R^{0}(\phi X,\xi )\xi  &=&\phi R(\phi X,\xi )\xi +\phi R^{\ast }(\phi
X,\xi )\xi   \notag \\
&&+\phi (\nabla _{\xi }^{\ast }A)\phi X+\phi \,(\nabla _{\xi }A^{\ast })\phi
X  \label{CM2} \\
&&+(AA^{\ast }+A^{\ast }A)X.  \notag
\end{eqnarray}%
Substracting (\ref{CM1}) from (\ref{CM2}) \ we get 
\begin{eqnarray}
4(R^{0}(X,\xi )\xi -\phi R^{0}(\phi X,\xi )\xi ) &=&R(X,\xi )\xi -\phi
R(\phi X,\xi )\xi +R^{\ast }(X,\xi )\xi -\phi R^{\ast }(\phi X,\xi )\xi  
\notag \\
&&+(\nabla _{\xi }^{\ast }A)X+\,(\nabla _{\xi }A^{\ast })X-\phi (\nabla
_{\xi }^{\ast }A)\phi X-\phi \,(\nabla _{\xi }A^{\ast })\phi X  \label{CM3}
\\
&&-2(AA^{\ast }+A^{\ast }A)X.  \notag
\end{eqnarray}%
Other hand, \ using (\ref{AA2}) and (\ref{b4}) we conclude that 
\begin{equation}
-2(AA^{\ast }+A^{\ast }A)=4((R^{0}(X,\xi )\xi -\phi R^{0}(\phi X,\xi )\xi )+2%
\mathcal{A}^{2}+2(\mathcal{A}^{\ast })^{2}.  \label{CMa4}
\end{equation}%
Using (\ref{CMa4}) in (\ref{CM3}) we have 
\begin{eqnarray}
0 &=&R(X,\xi )\xi -\phi R(\phi X,\xi )\xi +R^{\ast }(X,\xi )\xi -\phi
R^{\ast }(\phi X,\xi )\xi   \notag \\
&&+(\nabla _{\xi }^{\ast }\mathcal{A})X+\,(\nabla _{\xi }^{\ast }\mathcal{A}%
)X-\phi (\nabla _{\xi }^{\ast }\mathcal{A})\phi X-\phi \,(\nabla _{\xi
}^{\ast }\mathcal{A})\phi X  \label{CM5} \\
&&+2(\mathcal{A}^{2}+(\mathcal{A}^{\ast })^{2}).  \notag
\end{eqnarray}%
Beacause of Corollary \ref{KLM} we have 
\begin{equation}
\phi (\nabla _{\xi }^{\ast }\mathcal{A})\phi X+\phi \,(\nabla _{\xi }^{\ast }%
\mathcal{A})\phi X=(\nabla _{\xi }\mathcal{A})X+\,(\nabla _{\xi }^{\ast }%
\mathcal{A}^{\ast })X-g(\xi ,\nabla _{\xi }\mathcal{A}+\nabla _{\xi }^{\ast }%
\mathcal{A}^{\ast }X)\xi .  \label{CM6}
\end{equation}%
A short calculation leads to%
\begin{equation}
g(\xi ,\nabla _{\xi }\mathcal{A}+\nabla _{\xi }^{\ast }\mathcal{A}^{\ast
}X)=0.  \label{CM7}
\end{equation}%
We finally find 
\begin{equation}
\phi (\nabla _{\xi }^{\ast }\mathcal{A})\phi X+\phi \,(\nabla _{\xi }%
\mathcal{A}^{\ast })\phi X=(\nabla _{\xi }\mathcal{A})X+\,(\nabla _{\xi
}^{\ast }\mathcal{A}^{\ast }X)X.  \label{CM8}
\end{equation}%
Combining (\ref{CM5}) with (\ref{CM8}) we get%
\begin{eqnarray}
0 &=&R(X,\xi )\xi -\phi R(\phi X,\xi )\xi +R^{\ast }(X,\xi )\xi -\phi
R^{\ast }(\phi X,\xi )\xi   \label{CM9} \\
&&+(\nabla _{\xi }^{\ast }(\mathcal{A}X-\mathcal{A}^{\ast }))X-\,(\nabla
_{\xi }(\mathcal{A}-\mathcal{A}^{\ast }))X  \notag \\
&&+2(\mathcal{A}^{2}+(\mathcal{A}^{\ast })^{2}).  \notag
\end{eqnarray}%
Using the equality $A-A^{\ast }=2\mathcal{K}_{\xi }$ in (\ref{CM9}) we
obtain (\ref{RZETAZETA}).

Taking into account $\phi $-basis and (\ref{RZETAZETA}), we readily find 
\begin{equation*}
S(\xi ,\xi )+S^{\ast }(\xi ,\xi )=-tr(\mathcal{A}^{2}+(\mathcal{A}^{\ast
})^{2}).
\end{equation*}
\end{proof}

\end{document}